\documentclass[10pt]{amsart}

\usepackage{amsfonts,amsmath,amssymb,amsthm,amscd,amsxtra}
\usepackage{enumerate,epsfig,verbatim}

\usepackage{amssymb,amstext,amsmath,amscd,amsthm,amsfonts,enumerate,graphicx,latexsym}
\usepackage[usenames]{color}

\usepackage{nicefrac}
\usepackage{array}

\usepackage{comment}

\usepackage{mathptmx}

\usepackage{amsfonts,amsmath,amssymb,amsthm,amscd,amsxtra}
\usepackage{enumerate,verbatim}
\usepackage{centernot}
\usepackage{todonotes}

\usepackage[T1]{fontenc}
\usepackage[usenames,dvipsnames]{pstricks}
\usepackage[mathscr]{eucal}
\usepackage[notcite,notref]{}
\usepackage{tikz}
\usetikzlibrary{matrix,quotes}
\usepackage[notcite, notref]{}
\usepackage[usenames,dvipsnames]{pstricks}
\usepackage[mathscr]{eucal}
\usepackage{amsfonts,amsmath,amssymb,amsthm,amscd,amsxtra}
\usepackage{enumerate,verbatim}
\usepackage[all,2cell,ps]{xy}
\usepackage{mathptmx}
\usepackage[notcite, notref]{}
\usepackage[pagebackref]{hyperref}
\usepackage{mathtools}
\usepackage{nicefrac}
\usepackage{array}
\usepackage{xcolor}
\usepackage[all,2cell,ps]{xy}
\usepackage{amsfonts}
\usepackage[margin=1.4in]{geometry}
\usepackage{color}
\usepackage[notcite,notref]{}
\usepackage{url}
\usepackage{datetime}
\usepackage{mathtools}
\usepackage[all,2cell,ps]{xy}
\usepackage{amssymb}
\usepackage{amsmath}
\usepackage{amsfonts}
\usepackage{amsmath}
\usepackage{amsthm}
\usepackage{amssymb}
\usepackage{amscd}
\usepackage{amsfonts}
\usepackage{amsxtra}     
\usepackage{epsfig}
\usepackage{verbatim}
\usepackage[all]{xypic}

\SelectTips{cm}{}

\def\latex/{{\protect\LaTeX}}
\def\latexe/{{\protect\LaTeXe}}
\def\amslatex/{{\protect\AmS-\protect\LaTeX}}
\def\tex/{{\protect\TeX}}
\def\amstex/{{\protect\AmS-\protect\TeX}}
\def\bibtex/{{Bib\protect\TeX}}
\def\makeindx/{\textit{MakeIndex}}

\usepackage{amsmath}
\usepackage{amsthm}
\usepackage{amssymb}
\usepackage{amscd}
\usepackage{amsxtra}     
\usepackage{epsfig}
\usepackage{verbatim}
\usepackage[all]{xypic}

\usepackage{enumerate}

\theoremstyle{plain} 

\newtheorem{thm}{Theorem}[section]

\newtheorem{prop}[thm]{Proposition}
\newtheorem{cor}[thm]{Corollary}

\theoremstyle{definition}
\newtheorem{chunk}[thm]{\hspace*{-1.065ex}\bf}

\newtheorem{eg}[thm]{Example}

\newtheorem{rmk}[thm]{Remark}

\theoremstyle{remark}

\newtheorem*{claim*}{Claim}

\newcommand{\bZ}{\mathbb{Z}}

\newcommand{\fm}{\mathfrak{m}}

\newcommand{\ltensor}{\tensor^{\bf{L}}}

\newcommand{\CC}{\mathbb{C}}

\newcommand{\ZZ}{\mathbb{Z}}

\newcommand{\tensor}{\otimes}
\DeclareMathOperator{\id}{id}

\newcommand{\fp}{\mathfrak{p}}
\newcommand{\fq}{\mathfrak{q}}
\newcommand{\fr}{\mathfrak{r}}

 \DeclareMathOperator{\Tor}{Tor}

\DeclareMathOperator{\Ho}{H}
\DeclareMathOperator{\CI}{\textnormal{CI-dim}}

 \DeclareMathOperator{\Supp}{Supp}
 \DeclareMathOperator{\Spec}{Spec}

 \DeclareMathOperator{\pd}{pd}
 
 \DeclareMathOperator{\height}{height}
 \DeclareMathOperator{\length}{length}

 \DeclareMathOperator{\depth}{depth}

\DeclareMathOperator{\U}{\operatorname{\mathsf{U}}}

\newcommand{\ul}{\underline}

\newcommand{\Ann}{\textup{Ann}}

\def\urltilda{\kern -.15em\lower .7ex\hbox{\~{}}\kern
  .04em}\def\urldot{\kern -.10em.\kern -.10em}\def\urlhttp{http\kern
  -.10em\lower -.1ex\hbox{:}\kern -.12em\lower 0ex\hbox{/}\kern
  -.18em\lower 0ex\hbox{/}} 

\newcommand{\bb}{\left[ \begin{smallmatrix}}
\newcommand{\eb}{\end{smallmatrix} \right]}

\newcommand{\rH}{\mathrm{H}}
\newcommand{\rK}{\mathrm{K}}
\newcommand{\rV}{\mathrm{V}}
\begin{document}

\title[On the depth and reflexivity of tensor products]{On the depth and reflexivity of tensor products}

\author[O. Celikbas]{Olgur Celikbas}
\address{Olgur Celikbas\\
Department of Mathematics \\
West Virginia University\\
Morgantown, WV 26506-6310, U.S.A}
\email{olgur.celikbas@math.wvu.edu}

\author[U. Le]{Uyen Le}
\address{Uyen Le\\
Department of Mathematics \\
West Virginia University\\
Morgantown, WV 26506-6310, U.S.A}
\email{hle1@mix.wvu.edu}

\author[H. Matsui]{Hiroki Matsui} 
\address{Hiroki Matsui\\Graduate School of Mathematical Sciences\\ University of Tokyo, 3-8-1 Komaba, Meguro-ku, Tokyo 153-8914, Japan}
\email{mhiroki@ms.u-tokyo.ac.jp}

\thanks{Matsui was partly supported by JSPS Grant-in-Aid for JSPS Fellows 19J00158.}

\maketitle{}

\begin{abstract} In this paper we study the depth of tensor products of homologically finite complexes over commutative Noetherian local rings. As an application of our main result, we determine new conditions under which nonzero tensor products of finitely generated modules over hypersurface rings can be reflexive only if both of their factors are reflexive. 

A result of Asgharzadeh shows that nonzero symbolic powers of prime ideals in a local ring cannot have finite projective dimension, unless the ring in question is a domain. We make use of this fact in the appendix and consider the reflexivity of tensor products of prime ideals over hypersurface rings.
\end{abstract} 

\section{Introduction}

Throughout, $R$ denotes a commutative Noetherian local ring with unique maximal ideal $\fm$ and residue field $k$, and all $R$-modules are assumed to be finitely generated.

The results in this paper are motivated by the following beautiful result of Huneke and Wiegand; see  \cite[1.3 and 4.4]{GO}, \cite[1.1]{Onex}, \cite[2.7]{HW1}, and \cite[1.9]{HW2}.

\begin{thm} \label{srt} (Huneke and Wiegand  \cite{HW1, HW2}) Let $R$ be a local hypersurface ring, and let $M$ and $N$ be nonzero finitely generated $R$-modules. Assume $N$ has rank and $M\otimes_RN$ is reflexive. Then $\Tor_i^R(M,N)=0$ for all $i\geq 1$, $M$ is reflexive, $N$ is torsion-free, $\Supp_R(N)=\Spec(R)$, and $\pd_R(M)<\infty$ or $\pd_R(N)<\infty$.
\end{thm}

It has been an open problem for quite some time whether or not the module $N$ in Theorem \ref{srt} must also be reflexive; see \cite{GO, HW1, CHW} for the details. In 2019 Celikbas and Takahashi \cite{celikbas2018second} constructed examples settling this query; one of their examples is the following:

\begin{eg} \label{mainex} (\cite[2.5]{celikbas2018second}) Let $R=\CC[\!|x,y,z,w]\!]/(xy)$, $M=R/(x)$ and let $N$ be the Auslander transpose of $R/\fp$, where $\fp=(y,z,w) \in \Spec(R)$. Then $R$ is a reduced local hypersurface ring and $\pd_R(N)<\infty$ (so that $N$ has rank). Moreover, $M$ and $M \otimes_R N$ are both reflexive, but $N$ is not reflexive.
\end{eg}

In this paper, motivated by Theorem \ref{srt} and Example \ref{mainex}, we study the depth of tensor products and determine new conditions that force both of the modules considered in Theorem \ref{srt} to be reflexive. To faciliate the discussion, let us note, in Example \ref{mainex}, the sequence $\{y,z,w\}$ is $M$-regular, but it is not $R$-regular. We prove, if such sequences do not exist locally in the support of the module $M$ considered in Theorem \ref{srt}, then both of the modules in question must be reflexive. More precisely, we prove:

\begin{thm} \label{mainthmintro} Let $R$ be a local hypersurface ring, and let $M$ and $N$ be nonzero $R$-modules such that $M\otimes_RN$ is reflexive. Assume the following conditions hold:
\begin{enumerate}[\rm(i)]
\item $N$ has rank (e.g., $\pd_R(N)<\infty$).
\item Each $M_{\fp}$-regular sequence is $R_{\fp}$-regular for all $\fp \in \Supp_R(M)$. 
\end{enumerate}
Then $M$ and $N$ are both reflexive.
\end{thm}

It is worth pointing out that the condition in part (ii) of Theorem \ref{mainthmintro} holds provided that the module $M$ in question has full support; see Corollary \ref{corpropnew} (note the module $M$ considered in Example \ref{mainex} does not have full support). On the other hand, there are examples of modules -- without full support -- that satisfy the aforementioned condition of Theorem \ref{mainthmintro}; see Examples \ref{Hiroki} and \ref{exson}.

In Section 3 we establish Theorem \ref{mainthmintro} as a consequence of our main result, namely Theorem \ref{mainthm}, which concerns the depth of (derived) tensor products of homologically finite complexes that have finite complete intersection dimension over local rings. The proof of Theorem \ref{mainthmintro} relies upon a relation between the condition in part (ii) of Theorem \ref{mainthmintro} and a certain depth inequality, which does not hold for the module $M$ in Example \ref{mainex}; see Corollary \ref{cor2}, Example \ref{ref}, and Proposition \ref{propnew}. 

In Section 4 we compare Theorem \ref{mainthmintro} with the main result of \cite{Onex}, in which the reflexivity of tensor products of modules under the setting of Theorem \ref{srt} is also studied. We give examples and highlight that the condition we consider in part (ii) of Theorem \ref{mainthmintro} is independent of the main tool used in \cite{Onex}; see the examples and the first paragraph in Section 4.

In the appendix we give an application of Theorem \ref{srt}: we prove that, if the tensor product of two prime ideals is reflexive over a hypersurface ring that is not a domain, then both of the primes considered must be minimal. In fact, due to the work of Asgharzadeh \cite{Mohsen}, we are able to state our result in terms of the tensor product of symbolic powers of prime ideals; see Remark \ref{surp} and Corollary \ref{appcor1}. 

\section{Preliminaries} 

We start by recording some definitions and preliminary results that are needed for our arguments.

\begin{chunk} \label{cx1} \textbf{Complexes (\cite{Larsbook}).} Throughout, by an $R$-complex $X$, we mean a chain complex of $R$-modules which has homological differentials $\partial^X_i\colon X_i\to X_{i-1}$, and which is homologically finite, i.e., $\rH_i(X)=0$ for all $|i|\gg 0$ and each $\rH_i(X)$ is a finitely generated $R$-module.

If $X$ is a (not necessarily homologically finite) $R$-complex, we set:
$$
\sup X = \sup \{i \in \bZ \mid \rH_i(X) \neq 0\}
\mbox{ and }
\inf X = \inf \{i \in \bZ \mid \rH_i(X) \neq 0\}.
$$
Note we have that $\sup X = -\infty$ if and only if $\rH(X)=0$ if and only if $\inf X = \infty$. 
\end{chunk}

\begin{chunk} \label{cx} 
If $X$ and $Y$ are $R$-complexes, then it follows that $\inf(X\ltensor_RY)=\inf(X)+\inf(Y)$; see \cite[(A.4.11) and (A.4.16)]{Larsbook}. Here, $X \ltensor_R Y$ denotes the derived tensor product of $X$ and $Y$.
\end{chunk}

\begin{chunk} \textbf{Annihilator and Support (\cite[A.8.4]{Larsbook}).} \label{Support} 
  The \emph{annihilator} and \emph{support} of an $R$-complex $X$ is:
\begin{align*} \notag 
\Ann_R(X) =\bigcap_{i\in \ZZ} \Ann_R(\Ho_{i}(X))  \text{ and }
\Supp_R(X)=\bigcup_{i\in \ZZ}\Supp_R(\Ho_{i}(X)).
\end{align*}
Note that the equality $\Supp_R(X) = \rV(\Ann_R(X))$ holds.
\end{chunk}

\begin{chunk} \label{dcx1} \textbf{Depth of complexes (\cite{Larsbook, I}).}
Let $X$ be an $R$-complex, $I$ an ideal of $R$, and let $\ul{x}=x_1, \ldots, x_n$ be a generating set of $I$.
Then the {\it $I$-depth} of $X$ is defined as:
$$
\depth_R(I, X) = n - \sup(\rK(\ul{x}) \ltensor_R X).
$$
Here $\rK(\ul{x})$ is the Koszul complex on $\ul{x}$; see \cite[\S 2]{I} and \cite[(A.6.1)]{Larsbook}.  It is
known that this definition is independent of the choice of generators of $I$ \cite[1.3]{I}. It follows that $-\infty < \depth_R(I, X) \le n-\sup X$.

We set $\depth_R (X)= \depth_R(\fm, X)$. Then, by our convention for complexes, $\depth_R(X)$ is finite provided that $\Ho(X)
\neq 0$; see \cite[Observation on page 549]{I}. Note also that $\depth_R(0)=\infty$.
\end{chunk}

The following facts play an important role in the proof of Theorem \ref{mainthm}.

\begin{chunk} \label{dcx} Let $X$ be an $R$-complex and $I$ an ideal of $R$. Then the following hold:

\begin{enumerate}[\rm(i)]
\item $\depth_{R_\fp}(X_\fp) \ge \depth_{R_\fq}(X_\fq) - \dim( R_{\fq} / \fp R_{\fq})$ for each $\fp, \fq \in \Spec(R)$ with $\fp \subseteq \fq$; see \cite[(A.6.2)]{Larsbook}. 
\item $\depth_R(I, X) = \inf \{\depth_{R_\fp}(X_\fp) \mid \fp \in \rV(I)\}$; see \cite[2.10]{FS}. 
\item $\depth_R(I, X) = \depth_R(\sqrt{I}, X)$.
\item $\depth_R(I+\Ann_R(X), X)= \depth_R(I, X)$.
\end{enumerate}
Note that, as $\rV(I)=\rV(\sqrt{I})$ and $\rV\big(I+\Ann_R(X)\big)=\rV(I) \cap \Supp_R(X)$, part (ii) yields parts (iii) and (iv).
\end{chunk}

\begin{chunk} \textbf{Serre's condition for complexes (\cite{GO}).} \label{sc} Let $X$ be an $R$-complex and let $n\geq 0$ be an integer. Then $X$ is said to satisfy  {\it Serre's condition} $(S_n)$ if the following inequality holds for each prime ideal $\fp$ of $R$:
$$
\depth_{R_\fp}(X_\fp) + \inf (X_\fp) \ge \min\{n, \dim(R_{\fp})\}.
$$
\end{chunk}

\begin{chunk} \textbf{Complete intersection dimension of complexes (\cite{AGP, Sean}).} \label{CI1} Let $X$ be an $R$-complex. A diagram of local ring maps $R \to R' \twoheadleftarrow S$ is called a \emph{quasi-deformation} provided that $R \to R'$ is flat and the kernel of the surjection $R' \twoheadleftarrow S$ is generated by a regular sequence on $S$. 
The \emph{complete intersection dimension} of $X$ is defined as: 
\[
\CI_R(X)=\inf\{ \pd_S( X\ltensor_R R') -\pd_S(R'): R \to R' \twoheadleftarrow S \text{ is a quasi-deformation}\}.
\]
\end{chunk}

Some facts about the complete intersection dimension are recorded next:

\begin{chunk}\label{CI} Let $X$ be an $R$-complex. Then the following hold:
\begin{enumerate}[\rm(i)]
\item $\CI_R(X) \in \{-\infty\} \cup \ZZ \cup \{\infty\}$; see \cite[3.2.1]{Sean}.
\item  $\CI_R(X)=-\infty$ if and only if $\rH(X)=0$; see \cite[3.2.2]{Sean}. 
\item $\inf(X)\leq \sup(X) \leq \CI_R(X)$; see \cite[3.3]{Sean}.
\item If $\CI_R(X)<\infty$, then $\CI_R(X)=\depth(R)-\depth_R(X)$; see \cite[3.3]{Sean}.
\item If $R$ is a complete intersection, then $\CI_R(X)<\infty$; see \cite[3.5]{Sean}.
\end{enumerate}
\end{chunk}

\begin{chunk} \textbf{Derived Depth Formula.} \label{DF} Let $X$ and $Y$ be $R$-complexes. If $\CI_R(X)<\infty$ and $X \ltensor_R Y$ is bounded, i.e., $\Tor_{i}^{R}(X,Y)=0$ for all $i\gg 0$, then the equality $\depth_R(X) + \depth_R(Y) = \depth(R) + \depth_R(X \ltensor_R Y)$ holds, i.e., the pair $(X,Y)$ satisfies the \emph{derived depth formula}; see \cite[4.4]{CJ}. 
\end{chunk}

We need a few arguments from the proof of \cite[3.1]{GO} to establish our main result. In the following, for the sake of completeness, we include the arguments we need, along with a few additional details that are not explicitly stated in \cite{GO}.

\begin{chunk} (\cite[see the proof of 3.1]{GO}) \label{eski} Let $X$ and $Y$ be $R$-complexes such that $\rH(X)\neq 0 \neq \rH(Y)$. Assume $\CI_R(X)<\infty$. Assume further $X \ltensor_R Y$ is bounded and satisfies $(S_{n})$ for some $n\geq 0$. 

Let $\fp \in \Supp_R(Y)$. We proceed and look at $\depth_{R_{\fp}}(Y_{\fp}) +\inf(Y_\fp)$. We pick a minimal prime ideal $\fq$ of $\fp + \Ann_R(X)$ and consider the following three cases separately: $\dim(R_\fq) \le n$, $\dim (R_\fq) > n$, and  $\fp \in \Supp_R(X)$. Note that, by the choice, we have that $\fq \in \Supp_R(X \ltensor_R Y)$.

\emph{Case 1}. Assume $\dim(R_{\fq})\leq n$. Then it follows:
\begin{align} \notag{}
\depth_{R_{\fp}}(Y_{\fp}) +\inf(Y_\fp)&=\depth_{(R_{\fq})_{\fp R_{\fq}}}\left((Y_{\fq})_{\fp R_{\fq}}\right) +\inf((Y_\fq)_{\fp R_\fq})  &\\ \notag{}
& \geq \big[\depth_{R_{\fq}}(Y_{\fq}) -\dim(R_{\fq}/\fp R_{\fq})\big]+\inf(Y_\fq) &\\ \notag{} 
& = \big[\depth_{R_{\fq}}(X_\fq \ltensor_{R_\fq} Y_\fq)+\depth(R_{\fq})-\depth_{R_{\fq}}(X_{\fq})\big] -\dim(R_{\fq}/\fp R_{\fq})+\inf(Y_\fq) &\\ \notag{} 
& = \big[\depth_{R_{\fq}}(X_\fq \ltensor_{R_\fq} Y_\fq)+\CI_{R_{\fq}}(X_{\fq})\big] -\dim(R_{\fq}/\fp R_{\fq})+\inf(Y_\fq) &\\ \tag{\ref{eski}.1} 
& \geq \big[\depth_{R_{\fq}}(X_\fq \ltensor_{R_\fq} Y_\fq)+ \inf(X_\fq) \big] -\dim(R_{\fq}/\fp R_{\fq})+\inf(Y_\fq) &\\ \notag{} 
& = \depth_{R_{\fq}}(X_\fq \ltensor_{R_\fq} Y_\fq) +\inf(X_\fq \ltensor_{R_\fq} Y_\fq) -\dim(R_{\fq}/\fp R_{\fq}) &\\ \notag{} & \geq \min\{n, \dim(R_{\fq})\} - \dim(R_{\fq}/\fp R_{\fq}) \notag{} &\\ \notag{}  &\geq \dim(R_{\fq})-  \dim(R_{\fq}/\fp R_{\fq}) &\\  \notag{} &\geq \dim(R_{\fp}) &\\ \notag{} 
& \geq \min\{n, \dim(R_{\fp})\}. \notag{}
\end{align}
Here, in (\ref{eski}.1), the first inequality follows from \ref{dcx}(i) and the definition of inf, the second inequality follows from \ref{CI}(iii), and the third inequaliy is due to the fact that $X \ltensor_R Y$ satisfies $(S_{n})$; note that the other inequalities are standard. For the equalities in (\ref{eski}.1), the first one is due to localization, the second one follows from \ref{DF}, the third one is due to \ref{CI}(iv), and the fourth one can be obtained by \ref{cx}.

\emph{Case 2}. Assume $\dim(R_{\fq})> n$, and set $t=\depth_{R_{\fq}}(X_\fq \ltensor_{R_\fq} Y_\fq)+\inf(X_\fq \ltensor_{R_\fq} Y_\fq)$. Then it follows:
\begin{align} \notag{}
\depth_{R_{\fp}}(Y_{\fp}) + \inf(Y_\fp)& \geq \depth_{R_{\fq}}\big(Y_{\fq}) + \inf(Y_\fq) -\dim(R_{\fq}/\fp R_{\fq}\big) &\\ \notag{} 
& = \depth(R_{\fq})-\big(\depth_{R_{\fq}}(X_{\fq})+\inf(X_\fq)\big) + t-\dim(R_{\fq}/\fp R_{\fq}) &\\ \tag{\ref{eski}.2} 
&\geq \depth(R_{\fq})-\big(\depth_{R_{\fq}}(X_{\fq}) +\inf(X_\fq)\big)+ n - \dim(R_{\fq}/\fp R_{\fq})  &\\\notag{} 
&\geq \depth(R_{\fq})-(\depth_{R_{\fq}}(X_{\fq}) +\inf(X_\fq))+ n + \dim(R_{\fp}) -\dim(R_{\fq}).  
\end{align}
Here, in (\ref{eski}.2), the first inequality is due to \ref{dcx}(i) and the definition of inf, the second inequality is due to the fact that $X \ltensor_R Y$ satisfies $(S_{n})$ and $\dim(R_{\fq})> n$, and the third inequality is standard. Moreover, the equality in (\ref{eski}.2) follows from \ref{cx} and \ref{DF}.

\emph{Case 3}. Assume $\fp \in \Supp_R(X)$. Then one has $\fq = \fp$. Set $\dim(R_{\fp})=l$. If $l\leq n$, then it follows by Case 1 that $\depth_{R_{\fp}}(Y_{\fp}) +\inf(Y_\fp)\geq \min\{n, \dim(R_{\fp})\}$. Next assume $l> n$. Then it follows that:
\begin{align} \notag{}
\depth_{R_{\fp}}(Y_{\fp}) + \inf(Y_\fp) &\geq \depth(R_{\fp})-\big(\depth_{R_{\fp}}(X_{\fp}) +\inf(X_\fp)\big)+ n + \dim(R_{\fp}) -\dim(R_{\fp}) \notag{}\\
&= \CI_{R_{\fp}}(X_{\fp}) - \inf(X_\fp) + n  \tag{\ref{eski}.3} \\
&\geq n \geq \min\{n, \dim(R_\fp)\}. \notag{}
\end{align}
Here, in (\ref{eski}.3), the first inequality follows by Case 2, the second inequality is due to \ref{CI}(iii), and the equality is due to  \ref{CI}(iv). So, if $\fp \in \Supp_R(X)$, we have that $\depth_{R_\fp}(Y_\fp) + \inf Y_\fp \geq \min \{n, \dim(R_\fp)\}$. 
\end{chunk}

\section{Proof of the main theorem and corollaries}

We are now ready to prove the main result in this paper, namely Theorem \ref{mainthm}. The proof of the theorem is motivated by the results given in \ref{eski}, but the gist of our argument is different: the finishing touch of the proof of Theorem \ref{mainthm} relies upon an application of the properties stated in \ref{dcx}. 

We set, for $\fq\in \Spec(R)$, that $\U(\fq)=\{\fp \in \Spec(R): \fp \subseteq \fq \text{ where } \height_R(\fp)>0\}$.

\begin{thm} \label{mainthm} Let $R$ be a local ring, and let $X$ and $Y$ be $R$-complexes such that $\rH(X)\neq 0 \neq \rH(Y)$. Assume $m$ and $n$ are nonnegative integers and the following conditions hold:
\begin{enumerate}[\rm(i)]
\item $\CI_R(X)<\infty$.
\item $X \ltensor_R Y$ is bounded, i.e., $\Tor_{i}^{R}(X,Y)=0$ for all $i\gg 0$. 
\item $X \ltensor_R Y$ satisfies Serre's condition $(S_{n})$.
\item If $\fq \in \Supp_R(X \ltensor_R Y)$, then $\depth_{R_{\fq}}(\fp R_\fq, X_\fq) +\inf(X_\fq)\le \depth_{R_{\fq}}(\fp R_\fq, R_\fq)+m$ for all  $\fp \in \U(\fq)$.
\end{enumerate}
Then $Y$ satisfies Serre's condition $(S_{n-m})$.
\end{thm}

\begin{proof} Let $\fp \in \Supp_R(Y)$. We want to show that the following inequality holds:
\begin{align} \tag{\ref{mainthm}.1}
\depth_{R_\fp}(Y_\fp) + \inf (Y_\fp) \ge \min \{n -m, \dim(R_{\fp})\}.
\end{align}

If $\dim(R_{\fp})=0$, then (\ref{mainthm}.1) holds trivially. Moreover, if $\fp \in \Supp_R(X)$, then the inequality (\ref{mainthm}.1) holds by Case 3 of \ref{eski}. Hence we assume $\dim(R_{\fp})>0$, $\fp \notin \Supp_R(X)$, and pick a prime ideal $\fq$ of $R$ which is minimal over $\fp +\Ann_R(X)$. 
Then it follows that $\fq \in \Supp_R(X \ltensor_R Y)$. 

If $\dim(R_{\fq})\leq n$, then (\ref{mainthm}.1) holds by Case 1 of \ref{eski}. Hence, we further assume that $\dim(R_{\fq})> n$. Therefore, Case 2 of \ref{eski} yields:
\begin{align} \tag{\ref{mainthm}.2}
\begin{aligned}
\depth_{R_{\fp}}(Y_{\fp}) + \inf(Y_\fq) & \geq  \depth(R_{\fq})-(\depth_{R_{\fq}}(X_{\fq}) +\inf(X_\fq))+ n + \dim(R_{\fp}) -\dim(R_{\fq})&\\ \notag{} 
&=n + \dim(R_{\fp})-(\depth_{R_{\fq}}(X_{\fq}) +\inf(X_\fq)).
\end{aligned}
\end{align}

Now we suppose $\depth_{R_\fp} (Y_\fp) +\inf(Y_\fp) < \min\{n - m, \dim (R_\fp)\}$ and look for a contradiction. Note that we have $n-m > \depth_{R_\fp} (Y_\fp) +\inf(Y_\fp)$ and so (\ref{mainthm}.2) shows:
\begin{align} \tag{\ref{mainthm}.3} 
\dim(R_{\fp})<\depth_{R_{\fq}}(X_{\fq}) +\inf(X_\fq)-m.
\end{align}
Note that the following inequalities hold:
\[\begin{array}{rl}\tag{\ref{mainthm}.4}
\depth_{ R_{\fq} } (\fp R_{\fq}, X_{\fq} ) +\inf(X_\fq) -m & \le \depth_{ R_{\fq} } (\fp R_{\fq}, R_{\fq} ) \\ & \le \depth((R_{\fq})_{\fp R_{\fq}}) \\ & = \depth(R_{\fp}) \\ & \le \dim(R_{\fp})  \\ & < \depth_{R_{\fq}}(X_{\fq}) +\inf(X_\fq) -m.
\end{array}\] 
In (\ref{mainthm}.4), the first inequality is due to the hypothesis (iv) since $\fq \in \Supp_R(X \ltensor_R Y)$ and $\fp \in \U(\fq)$. Moreover, the second inequality of (\ref{mainthm}.4) follows from \ref{dcx}(ii), the third one is by \cite[1.2.12]{BH}, and the forth one is due to (\ref{mainthm}.3). Hence (\ref{mainthm}.4) gives:
\begin{equation}\tag{\ref{mainthm}.5}
\depth_{ R_{\fq} } (\fp R_{\fq}, X_{\fq} )<\depth_{R_{\fq}}(X_{\fq}).
\end{equation}

On the other hand, we have:
\[\begin{array}{rl}\tag{\ref{mainthm}.6}
\depth_{R_{\fq}}(X_{\fq}) & =  \depth_{R_{\fq}}(\fq R_{\fq}, X_{\fq})\\ & =  
\depth_{R_{\fq}}\Big(\sqrt{\fp R_{\fq}+\Ann_{R_{\fq}}(X_{\fq})}, X_{\fq}\Big)\\
&=\depth_{R_{\fq}}(\fp R_{\fq}+\Ann_{R_{\fq}}(X_{\fq}), X_{\fq})  \\
&=\depth_{R_{\fq}}(\fp R_{\fq}, X_{\fq})\\
\end{array}\] 
In (\ref{mainthm}.6), the first equality follows from \ref{dcx1}, the second one is due to the fact that $\fp R_{\fq}+\Ann_{R_{\fq}}(X_{\fq})$ is $\fq R_{\fq}$-primary, and the last two equalities follow from \ref{dcx}(iii) and \ref{dcx}(iv), respectively. 

Consequently, in view of (\ref{mainthm}.5) and (\ref{mainthm}.6), we obtain a contradiction. This contradiction implies that the inequality (\ref{mainthm}.1) holds, and hence completes the proof.
\end{proof}

We proceed by recording some consequences of Theorem \ref{mainthm}. 

\begin{cor} \label{cor1} Let $R$ be a local ring, $M$ and $N$ be finitely generated $R$-modules, and let $m$ and $n$ be nonnegative integers. 
Assume the following hold:
\begin{enumerate}[\rm(i)]
\item $\CI_R(M)<\infty$.
\item $\Tor_i^R(M,N)=0$ for all $i\geq 1$.
\item $M\otimes_RN$ satisfies Serre's condition $(S_{n})$.
\item If $\fq \in \Supp_R(M \otimes_R N)$, then $\depth_{R_{\fq}}(\fp R_\fq, M_\fq) \le \depth_{R_{\fq}}(\fp R_\fq, R_\fq)+m$ for all  $\fp \in \U(\fq)$.
\end{enumerate}
Then $N$ satisfies Serre's condition $(S_{n-m})$.
\end{cor}

\begin{proof} Note that $M \ltensor_R N \cong M \otimes_R N$ if and only if $\Tor_i^R(M,N)=0$ for all $i\geq 1$. Therefore, it follows by Theorem \ref{mainthm} that $N$ satisfies Serre's condition $(S_{n-m})$.
\end{proof}

\begin{rmk} In \cite{Omo} one can find further results concerning the depth inequality stated in part (iv) of Corollary \ref{cor1}; see also Proposition \ref{propnew}. In fact, when $m=1$, it is proved in \cite{Omo} that the aforementioned inequality always holds over hypersurface rings. More precisely, if $R$ is a hypersurface ring, $I$ an ideal of $R$, and $M$ is a non-zero torsion-free $R$-module which is generically free, then it follows from a result in \cite{Omo} that $\depth(I, M) \le \depth_R(I, R) +1$.\pushQED{\qed} 
\qedhere 
\popQED
\end{rmk}

Modules over hypersurface rings are reflexive if and only if they satisfy Serre's condition $(S_2)$; see, for example, \cite[2.5]{GO}. This fact is used in the next corollary of Theorem \ref{mainthm}.

\begin{cor} \label{cor2} Let $R$ be a local hypersurface ring, and let $M$ and $N$ be nonzero $R$-modules. Assume the following conditions hold:
\begin{enumerate}[\rm(i)]
\item $N$ has rank.
\item If $\fq \in \Supp_R(M)$, then $\depth_{R_{\fq}}(\fp R_\fq, M_\fq) \le \height_R(\fp)$ for all $\fp \in \U(\fq)$.
\end{enumerate}
If $M\otimes_RN$ is reflexive, then $M$ and $N$ are both reflexive. 
\end{cor}

\begin{proof} Assume $M\otimes_RN$ is reflexive. Then it follows from Theorem \ref{srt} that $\Tor_i^R(M,N)=0$ for all $i\geq 1$ and $M$ is reflexive. Moreover, since $R$ is Cohen-Macaulay, the equality $\depth_{R_{\fq}}(\fp R_\fq, R_\fq)=\height_R(\fp)$ holds for each $\fp, \fq \in \Spec(R)$ with $\fp \subseteq \fq$. Note also $\CI_R(M)<\infty$ as $R$ is a hypersurface; see \ref{CI}(v). Therefore, setting $m=0$ and $n=2$, we conclude from Corollary \ref{cor1} that $N$ is reflexive. 
\end{proof}

Recall that the module $N$ in Example \ref{mainex} is not reflexive. Hence, it is worth pointing out that the depth inequality in part (ii) of Corollary \ref{cor2} does not hold for the module $M$ in the example.

\begin{eg} \label{ref} Let $R$, $M$ and $N$ be as in Example \ref{mainex}, i.e., $R=\CC[\!|x,y,z,w]\!]/(xy)$, $M=R/(x)$ and let $N$ be the Auslander transpose of $R/\fp$, where $\fp=(y,z,w) \in \Spec(R)$. Let $\fq=\fm$. Then it follows that $\depth_{R_{\fq}}(\fp R_\fq, M_\fq)=\depth_R(\fp, M)=3 > \height_R(\fp)=2$.\pushQED{\qed} 
\qedhere 
\popQED
\end{eg}

Now our aim is to establish Theorem \ref{mainthmintro}, advertised in the introduction. First we prove the following general result which seems to be of independent interest.

\begin{prop} \label{propnew} Let $R$ be a local ring and let $M$ be a nonzero $R$-module. Then the following conditions are equivalent:
\begin{enumerate}[\rm(i)]
\item Each $M$-regular sequence is $R$-regular.
\item $\depth_{R}(I, M) \leq \depth_R(I,R)$ for each ideal $I$ of $R$.
\item $\depth_{R}(\fp, M) \leq \depth_R(\fp,R)$ for each $\fp \in \Spec(R)$.
\end{enumerate}
\end{prop}

\begin{proof} It follows by definition that $\rm{(i)} \Longrightarrow \rm{(ii)}  \Longrightarrow \rm{(iii)}$. So we proceed and show $\rm{(iii)} \Longrightarrow \rm{(ii)} \Longrightarrow \rm{(i)}$.

$\rm{(iii)} \Longrightarrow \rm{(ii)}$: Write $\sqrt{I}=\fp_1 \cap \ldots \cap \fp_n$ for some prime ideals $\fp_i$. Then it follows that
\begin{align}\notag{}
 \depth_R(I,M) \notag{} & =\depth_R(\sqrt{I}, M) \\& =\inf\{ \depth_R(\fp_1, M), \ldots, \depth_R(\fp_n, M) \} \notag{} \\ & \leq  \inf\{ \depth_R(\fp_1, R), \ldots, \depth_R(\fp_n, R) \} \tag{\ref{propnew}.1} \\& = \depth_R(\sqrt{I}, R) \notag{} \\& = \notag{} \depth_R(I,R).
\end{align}
Here in (\ref{propnew}.1), the first and the fourth equalities are due to \cite[1.2.10(b)]{BH} (see also \ref{dcx}(ii)), the second and third equalities are due to \cite[1.2.10(c)]{BH}, and the inequality follows by assumption.

$\rm{(ii)} \Longrightarrow \rm{(i)}$: Let $\underline{x}=x_1, \ldots, x_n \subseteq \fm$ be an $M$-regular sequence. We will show that this sequence is $R$-regular by induction on $n$. 

If $n=1$, then we have $1\leq \depth_R((x_1), M) \leq   \depth_R((x_1), R)$, which implies that $x_1$ is a non zero-divisor on $R$. Hence we assume $n\geq 2$. Then, by the induction hypothesis, it follows that $\underline{x}'=x_1, \ldots, x_{n-1}$ is $R$-regular. Thus, we have:
\begin{align}\notag{}
\notag{} \depth_R((x_n), M/ (\underline{x}') M)+(n-1) \notag{} & =  \depth_R((\underline{x}),M)  \notag{} \\ & \leq \depth_R((\underline{x}),R) \tag{\ref{propnew}.2} \\ \notag{} & = \depth_R((x_n), R/ (\underline{x}'))+(n-1)
\end{align}
Here in (\ref{propnew}.2), the equalities are due to \cite[1.2.10(d)]{BH}, while the inequality follows by assumption. Therefore, we have 
\begin{align*}
1\leq \depth_R((x_n), M/ (\underline{x}') M) \leq \depth_R((x_n), R/ (\underline{x}')R), 
\end{align*}
which implies that $x_n$ is a non zero-divisor on $R/ (\underline{x}') R$, as required.
\end{proof}

\begin{rmk} The equivalent conditions in Proposition \ref{propnew} hold if and only if, whenever $\underline{x}=x_1, \ldots, x_n $ is a sequence of elements in $\fm$ with $\Tor_1^R(M, R/\underline{x}R)=0$, it follows $\Tor_2^R(M, R/\underline{x}R)=0$; see \cite[2.2]{Kamal}.\pushQED{\qed} \qedhere \popQED
\end{rmk}

The following result is an immediate consequence of Proposition \ref{propnew}:

\begin{cor}\label{corpropnew} Let $R$ be a local ring, $I$ an ideal of $R$, and let $M$ be a nonzero $R$-module such that $\Supp_R(M)=\Spec(R)$. If $\depth_{R_{\fp}}(M_{\fp}) \leq \depth(R_{\fp})$ for all $\fp \in \Spec(R)$ (e.g., $R$ is Cohen-Macaulay, or $\CI_R(M)<\infty$), then each $M_{\fp}$-regular sequence is $R_{\fp}$-regular for all $\fp \in \Spec(R)$.
\end{cor}

\begin{proof} We have $\depth_R(I, M) = \inf \{\depth_{R_\fp}(M_\fp) \mid \fp \in \rV(I)\} \leq \inf \{\depth(R_\fp) \mid \fp \in \rV(I)\} = \depth(I, R)$; see \ref{dcx}(ii). Hence the result follows from Proposition \ref{propnew}.
\end{proof}

We are now ready to prove Theorem \ref{mainthmintro}. Let us first note a fact proved in \cite[2.6]{HW1}: if the module $M$ in Theorem \ref{mainthmintro} has full support, then $M$ and $N$ have full support so that a quick application of the depth formula shows that both $M$ and $N$ satisfy $(S_2)$, i.e., both $M$ and $N$ are reflexive; see also \ref{DF}, and \cite[1.3]{GO} for the details. Therefore, the gist of Theorem \ref{mainthmintro} is the case where $\Supp_R(M)\neq \Spec(R)$.

\begin{proof}[Proof of Theorem \ref{mainthmintro}] Let $\fq  \in \Supp_R(M)$. Then, since each $M_{\fq}$-regular sequence is $R_{\fq}$-regular by assumption, it follows from Proposition \ref{propnew} that $\depth_{R_{\fq}}(\fp R_\fq, M_\fq) \le \height_R(\fp)$ for all $\fp \in \Spec(R)$ with $\fp \subseteq \fq$. Therefore, we have that $\depth_{R_{\fq}}(\fp R_\fq, M_\fq) \le \height_R(\fp)$ for all $\fp \in \U(\fq)$. Consequently, Corollary \ref{cor2} implies that both $M$ and $N$ are reflexive.
\end{proof}

It is interesting to note that Theorem \ref{mainthmintro} (and also Theorem \ref{srt}) can fail over rings that are not hypersurfaces. For example, if $R=k[\![t^3,t^4, t^5]\!]$ and $N=(t^3, t^4)$, the canonical module of $R$, Huneke and Wiegand \cite[4.8]{HW1} constructs an $R$-module $M$ such that $M\otimes_RN$ is reflexive, but neither $M$ nor $N$ is reflexive; note that the hypotheses in parts (i) and (ii) of Theorem \ref{mainthmintro} hold for these modules $M$ and $N$. In \cite[2.1]{Con} one can find a similar example over a Gorenstein ring that is not a hypersurface.

\section{Further remarks on Theorem \ref{mainthmintro}}

Recall that an $R$-module $M$ is called \emph{Tor-rigid} provided that the following condition holds: for each $R$-module $N$ satisfying $\Tor_1^R(M,N)=0$, one has that $\Tor_2^R(M,N)=0$. Examples of Tor-rigid modules are abundant in the literature. For example, if $R$ is hypersurface, that is quotient of an unramified regular local ring, and $M$ is an $R$-module such that $\length_R(M)<\infty$ or $\pd_R(M)<\infty$, then $M$ is Tor-rigid; see  \cite[2.4]{HW1} and \cite[Theorem 3]{Li}, respectively. Tor-rigidity condition can impose certain restrictions on the ring in question. For example, Auslander \cite[4.3]{Au} proved that, if $M$ is a nonzero Tor-rigid module over a local ring $R$, then each $M$-regular sequence is an $R$-regular sequence. Note, this fact implies that the depth of a nonzero Tor-rigid module is always bounded by the depth of the ring considered.

In 2019 Celikbas, Matsui and Sadeghi \cite{Onex} examined the conclusion of Theorem \ref{srt} and studied the reflexivity of tensor products of modules over local hypersurface rings in terms of the Tor-rigidity. Their main result establishes the same conclusion of Theorem \ref{mainthmintro} for Tor-rigid modules. More precisely, the main result of \cite{Onex} shows that, if $M$ and $N$ are nonzero modules over a local hypersurface ring $R$ such that $M\otimes_RN$ is reflexive, $M$ is Tor-rigid, and $N$ has rank, then $M$ and $N$ are both reflexive; see Theorem \ref{srt} and \cite[3.1]{Onex}. Therefore, we next give examples and highlight that the Tor-rigidity condition and the condition we study in this paper, namely the condition in part (ii) of Theorem \ref{mainthmintro}, are independent of each other, in general. 

\begin{eg} \label{Hiroki} Let $R=k[\![x,y,z]\!]/(xy)$ and let $M=R/(x^2)$. Then it follows that $\Supp_R(M)\neq \Spec(R)$, $M$ is not Tor-rigid, and each $M_{\fp}$-regular sequence is $R_{\fp}$-regular for all $\fp\in \Supp_R(M)$. We justify these properties as follows:

(i) $\Supp_R(M)\neq \Spec(R)$: this is clear since $(y) \notin \Supp_R(M)$. In fact, since $\{(x), (x,y)\}$ is the set of all associated primes of $M$, it follows that $\Supp_R(M)=\rV\big((x)\big)\cup \rV\big((x,y)\big)$.

(ii) $M$ is not Tor-rigid: setting $N=R/(y)$, one can check that $\Tor_1^R(M,N)= 0\neq \Tor_2^R(M,N)$.

(iii) Each $M_{\fp}$-regular sequence is $R_{\fp}$-regular for all $\fp\in \Supp_R(M)$: note, to justify this claim, due to Proposition \ref{propnew}, we proceed to prove the following claim:
\[ \tag{\ref{Hiroki}.1}
\text{If } I \text{ is an ideal of } R \text{ such that } I \subseteq \fp\in \Supp_R(M), \text{ then } \depth_{R_{\fp}}(IR_{\fp}, M_{\fp}) \leq \height_{R_{\fp}}(IR_{\fp}).
\] 
Let $\fp \in \Supp_R(M)$ and let $I$ be an ideal of $R$ such that $I \subseteq \fp$. We look at the height of $\fp$, i.e., $\dim(R_{\fp})$.

\emph{Case 1}: Assume $\height_R(\fp)=0$. In this case the claim in (\ref{Hiroki}.1) holds as $\height_{R_{\fp}}(IR_{\fp}) \leq \dim(R_{\fp})$ and $ \depth_{R_{\fp}}(IR_{\fp}, M_{\fp}) \leq  \depth_{R_{\fp}}(M_{\fp}) \leq \dim(R_{\fp})$. 

\emph{Case 2}: Assume $\height_R(\fp)=1$. We first consider the case where $\fp=(x,y)$. As $\fp$ is an associated prime of $M$, it follows that $\depth_{R_{\fp}}(IR_{\fp}, M_{\fp}) \leq  \depth_{R_{\fp}}(M_{\fp})=0$, and so the claim in (\ref{Hiroki}.1) holds. 

Next, we consider the case where $\fp\neq (x,y)$. Note, as $\fp \in \Supp_R(M)$, we have that $(x) \subseteq \fp$. Moreover, in case $I \subseteq \fr \subsetneqq \fp$ for some $\fr \in \Spec(R)$, one can observe that $I\subseteq (x)$. As $\depth_{R_{(x)}}(M_{(x)})=0$, the aforementioned observation and \ref{dcx}(ii) yield:
\[ \tag{\ref{Hiroki}.2}
\depth_{R_{\fp}}(IR_{\fp}, M_{\fp}) = \inf \{\depth_{R_\fq}(M_\fq) \mid I \subseteq \fq \subseteq \fp\} = 
\begin{cases}
 \depth_{R_{\fp}}(M_{\fp})  \text{ if } I \nsubseteq (x)\\
0 \; \; \; \; \; \; \; \; \; \; \; \;  \; \; \; \; \; \; \; \; \; \text{ if } I \subseteq (x)\\
\end{cases}
\] 
As the equalities in (\ref{Hiroki}.2) also hold when $M$ is replaced with $R$, the claim in (\ref{Hiroki}.1) holds. 

\emph{Case 3}: Assume $\height_R(\fp)=2$, i.e., $\fp=\fm$. As $\depth_{R}(I, M) \leq  \depth_{R}(M)=1$, to establish the claim in (\ref{Hiroki}.1), it suffices to assume $\height_R(I)=0$ and show that $\depth_{R}(I, M)=0$. We observe, as each element of $I$ is a zero-divisor on $R$, that $I \subseteq (x)$ or $I \subseteq (y)$. Thus, we have $I \subseteq (x,y)$ and hence:
\[
\depth_{R}(I, M) = \inf \{\depth_{R_\fr}(M_\fr) \mid \fr \in \rV(I)\} \leq\depth_{R_{(x,y)}}(M_{(x,y)})=0. 
\] 
This completes the proof of Case 3.\pushQED{\qed} 
\qedhere 
\popQED
\end{eg}

\begin{eg}\label{exson} Let $R=k[\![x,y,z,u]\!]/(xy)$ and let $M=N \oplus T$, where $N=R/(x)$ and $T$ is an $R$-module such that $\dim_R(T)=0$ and $\pd_R(T)=\infty$ (e.g., $T=k$). Then it follows that $\Supp_R(M)\neq \Spec(R)$, $M$ is Tor-rigid, and there is an $M_{\fp}$-regular sequence which is not $R_{\fp}$-regular for some $\fp\in \Supp_R(M)$. We justify these properties as follows:

(i) $\Supp_R(M)\neq \Spec(R)$: this is clear since $(y) \notin \Supp_R(M)$. 

(ii) $M$ is Tor-rigid: to see this assume $\Tor_1^R(M, X)=0$ for some $R$-module $X$. Then $\Tor_1^R(T, X)=0$, and since $T$ is Tor-rigid \cite[2.4]{HW1}, we have that $\Tor_i^R(T, X)=0$ for all $i\geq 1$. This implies $X$ is free and hence $\Tor_i^R(M, X)=0$ for all $i\geq 1$; see \cite[2.5]{HW1} (or see \ref{CI}(v) and \ref{DF}).  Therefore, $M$ is Tor-rigid (note also that $N$ is not Tor-rigid since $\Tor_1^{R}(N, R/yR)=0 \neq R/\fp =\Tor_2^{R}(N, R/yR)$). 

(iii) There is an $M_{\fp}$-regular sequence which is not $R_{\fp}$-regular for some $\fp\in \Supp_R(M)$: for this part, let $\fp=(x,y)$. Then it follows that $\fp \in \Supp_R(M)$ and $M_{\fp} \cong N_{\fp}$. Hence, $y$ is a non zero-divisor on $M_{\fp}$. On the other hand, as $x\neq 0$, $y\neq 0$ and $xy=0$ in $R_{\fp}$, we see that $y$ is a zero-divisor on $R_{\fp}$. Thus, $\{y\}$ is an $M_{\fp}$-regular sequence which is not $R_{\fp}$-regular.
\pushQED{\qed} 
\qedhere 
\popQED
\end{eg}

The modules considered in Examples \ref{Hiroki} and \ref{exson} do not have full support. Next, in Example \ref{exson2}, we look at a module $M$ that has full support and observe, even for such a module, Tor-rigidity condition is distinct from the condition stated in part (ii) of Theorem \ref{mainthmintro}.

\begin{eg}\label{exson2} Let $R=k[\![x,y,z,u]\!]/(xu-yz)$ and let $M=(x,y)=\Omega \big(R/(x,y)\big)$. Then each $M_{\fp}$-regular sequence is $R_{\fp}$-regular because $\Supp_R(M)=\Spec(R)$; see Corollary \ref{corpropnew}. Furthermore, it follows that $\Tor^R_1(M,M)=0\neq k=\Tor^R_2(M,M)$, and hence $M$ is not Tor-rigid.\pushQED{\qed} 
\qedhere 
\popQED
\end{eg}

One can also construct examples similar to Example \ref{exson2} over rings that are not hypersurfaces.

\begin{eg} \label{exson0} Let $R$ be a Cohen-Macaulay local ring with canonical module $\omega$ such that $\omega \ncong R$; for example, $R=k[\![t^3,t^4, t^5]\!]$ and $\omega=(t^3, t^4)$. Then $M$ is not Tor-rigid; see, for example, \cite[4.13(i)]{HDT}. On the other hand, each $M_{\fp}$-regular sequence is $R_{\fp}$-regular for all $\fp \in \Spec(R)$; see Corollary \ref{corpropnew}.\pushQED{\qed} 
\qedhere 
\popQED
\end{eg}

It is well-known that the Tor-rigidity property does not localize, in general. For example, if $M$ is the module considered in Example \ref{exson}, then $M$ is Tor-rigid over $R$, but $M_{\fp}$ is not Tor-rigid over $R_{\fp}$. Furhermore, the same example also shows that the condition we consider in part (ii) of Theorem \ref{mainthmintro} does not localize in general, too. It seems worth summarizing these observations as a separate remark; see also Proposition \ref{propnew}.

\begin{rmk} \label{sonolsun} Let $R$ be a local ring and let $M$ be a nonzero $R$-module. Consider the following conditions.
\begin{enumerate}[\rm(i)] 
\item $M$ is Tor-rigid over $R$.
\item $M_{\fp}$ is Tor-rigid over $R_{\fp}$ for all $\fp \in \Supp_R(M)$.
\item Each $M$-regular sequence is $R$-regular.
\item Each $M_{\fp}$-regular sequence is $R_{\fp}$-regular for all $\fp \in \Supp_R(M)$.
\end{enumerate}
Then we have:
\vspace*{-9ex}

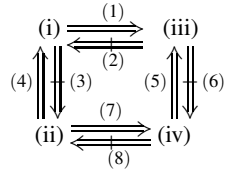
$$
{\small
\xymatrix@C=3em@R=3em{
& \\
\text{(i)} \ar@{=>}[r]|{\object@{}}^-{\;\;\;\; (1)}  \ar@{=>}[d] <0.8ex>|{\object@{|}}^-{ \;(3)}  & \text{\;\;(iii)} \ar@{=>}[l]<1.5ex>|{\object@{|}}^-{\;\;\;\; (2)}  \ar@{=>}[d]<1.5ex>|{\object@{|}}|{}^-{\; (6)}  &  \\ \text{(ii)} \ar@{=>}[u]<0.8ex>|{\object@{}}^-{\;\;\;\; (4)} \ar@{=>}[r]<0.6ex>|{\object@{}}^-{(7)}
& \text{(iv)} \ar@{=>}[l]<0.7ex>|{\object@{|}}^-{\;\;\;\; (8)}  \ar@{=>}[u] |{\object@{}}|{}^-{\;\;\;\; (5)} }} \\
$$

The implications in the above diagram can be justified as follows:

\noindent (1) and (7): see \cite[4.3]{Au}.\\
\noindent (2) and (8): see the module $M$ in Example \ref{exson2}.\\
\noindent (3): see the module $M$ in Example \ref{exson}.\\
\noindent (4) and (5): these follow by definition.\\
\noindent (6): in view of \cite[4.3]{Au}, see the module $M$ in Example \ref{exson}.
\end{rmk}

\appendix
\section{An application of Theorem \ref{srt}}

In this appendix, we give an application of Theorem \ref{srt}, and provide a criterion for tensor products of prime ideals to be reflexive over hypersurface rings. More precisely, we prove in Corollary \ref{appcor1} that, if $R$ be a hypersurface ring that is not a domain, and the tensor product of two prime ideals is refexive, then both of the primes considered must be minimal. Our result, which seems to be new, is based on the following observations of Asgharzadeh \cite[5.1 and 5.5]{Mohsen}; see also \cite[II.3.3]{PS2}.

\begin{rmk} \label{surp} Let $R$ be a commutative Noetherian ring, and let $\fp \in \Spec(R)$. Assume $\fp^{(n)}\neq 0$ for some $n\geq 1$, where $\fp^{(n)}= \fp^n R_{\fp} \cap R$ denotes the $n$th \emph{symbolic power} of $\fp$. Set $M=R/\fp^{(n)}$. 
\begin{enumerate}[\rm(i)]
\item Assume $M$ is Tor-rigid over $R$. Then each non zero-divisor on $M$ is a non zero-divisor on $R$ so that the canonical map $R \to R_{\fp}$ is injective; see Remark \ref{sonolsun}. Hence, $R$ is a domain if $R_{\fp}$ is a domain. 
\item Assume $\pd_R(M)<\infty$. Then each non zero-divisor on $M$ is a non zero-divisor on $R$ so that the canonical map $R \to R_{\fp}$ is injective; see \cite{R2}, \cite[6.2.3]{Robertsbook}. Also, $R$ is a domain as $R_{\fp}$ is regular.
\item Assume $\id_R(M)<\infty$. Then it follows that $R$ is Gorenstein \cite[II.5.3]{PS2} so that $\pd_R(M)<\infty$. Hence, part (ii) implies that $R$ is a domain.
\end{enumerate}
\end{rmk}

\begin{prop} \label{appp1} Let $R$ be a local hypersurface ring which is not a domain, and let $M$ be a nonzero $R$-module. Let $\fp\in \Spec(R)$ such that $\height_R(\fp)\geq 1$ and $\fp^{(n)}\neq 0$ for some $n\geq 1$. Set $N=\Omega^r\big(R/\fp^{(n)}\big)$ for some $r\geq 0$. Assume $M\otimes_RN$ is reflexive. Then it follows that $r\geq 1$, both $M$ and $N$ are reflexive, and $\pd_R(M)<\infty=\pd_R(N)$. Moreover, if $M$ is not free, then $M$ has rank at least two.
\end{prop}

\begin{proof} As $M\otimes_RN$ is a nonzero torsion-free $R$-module, we observe that neither $M$ nor $N$ can be torsion. This implies that $r\geq 1$. Note, since $\fp$ has positive height, it follows that $R/\fp$ is torsion, i.e., $R/\fp$ has rank zero. Thus $N$ has rank, and so the conclusions of Theorem \ref{srt} hold.

We know, by Theorem \ref{srt}, that $\pd_R(M)<\infty$ or $\pd_R(N)<\infty$. However, if $\pd_R(N)<\infty$, then $\pd_R(R/\fp^{(n)})<\infty$ and this forces $R$ to be a domain; see part (ii) of Remark \ref{surp}. Therefore, we have that $\pd_R(M)<\infty=\pd_R(N)$.

Notice, since both $M$ and $N$ have rank, both of these modules have full support. Consequently, Theorem \ref{srt} implies that both $M$ and $N$ are reflexive; see \cite[1.3]{GO}.

Now assume $M$ is not free. Then, since $M$ is reflexive, it follows that $\pd_R(M)\geq 3$. Hence, as $M$ is a second syzygy module, the syzygy theorem of Evans and Griffith \cite[1.1]{SP} (see also \cite{Andre}, \cite[9.5.6]{BH}, and \cite{Ogata}) forces $M$ to have rank at least two.
\end{proof}

The positive height assumption on the prime ideal considered in Proposition \ref{appp1} cannot be removed.

\begin{eg} \label{appeg1} Let $R=k[\![x,y]\!]/(xy)$, $\fp=(x)$, and let $\fq=(y)$. Then $\fp$ and $\fq$ are the minimal prime ideals of $R$. Set $M=R/(x^2)$ and $N=\Omega(R/\fp)$. Then $\pd_R(M)=\infty$, but $M\otimes_R N \cong N$ is a reflexive $R$-module.
\end{eg}

The next corollary of Proposition \ref{appp1} yields the criterion we seek concerning the tensor products of prime ideals over local hypersurface rings:

\begin{cor}  \label{appcor1} Let $R$ be a local hypersurface ring that is not a domain, and let $\fp, \fq \in \Spec(R)$. Assume $\fp^{(r)}\neq 0$ and $\fp^{(s)}\neq 0$ for some $r\geq 1$ and $s\geq 1$. If $\fp$ or $\fq$ has positive height, then $\fp^{(r)} \otimes_R \fq^{(s)}$ is not a reflexive $R$-module. Therefore, if  $\fp \otimes_R \fq$ is reflexive, then both $\fp$ and $\fq$ are minimal primes.
\end{cor}

In view of Corollary \ref{appcor1},  it is worth noting that the tensor product of two minimal prime ideals over a non-domain hypersurface ring may, or may not, be reflexive. 

\begin{eg} Let $R$, $\fp$ and $\fq$ be as in Example \ref{appeg1}. Then $\fp$ and $\fq$ are the minimal prime ideals of $R$. It follows that $\fp \cong R/(y)$ and $\fp \otimes_R\fp \cong \fp$ are reflexive $R$-modules. On the other hand, the tensor product $\fp \otimes_R \fq \cong k$ is not reflexive.
\end{eg}

\section*{Acknowledgements}
The authors thank Mohsen Asgharzadeh and Greg Piepmeyer for their comments and suggestions on an earlier version of the manuscript. We also thank 
Arash Sadeghi for pointing us the rank conclusion in Proposition \ref{appp1}.

\end{document}